\documentclass[12pt,reqno]{amsart}
\usepackage{amssymb,latexsym,graphicx,amscd}
\usepackage[utf8]{inputenc}
\usepackage[all]{xy}
\usepackage{epic,eepic}
\usepackage{cases}
\usepackage{bbm}         
\newtheorem{Thm}{Theorem}[section]
\newtheorem{Lem}[Thm]{Lemma}
\newtheorem{Cor}[Thm]{Corollary}
\newtheorem{Prop}[Thm]{Proposition}
\newtheorem{Conj}[Thm]{Conjecture}
\newtheorem{Qu}[Thm]{Question}

\theoremstyle{definition}

\newcommand{\K}{\mathbb{K}}
\newcommand{\Z}{\mathbb{Z}}

\newcommand{\N}{\mathbb{N}}
\newcommand{\C}{\mathbb{C}}

\newcommand{\df}{\colon}

\newcommand{\cA}{{\mathcal A}}

\newcommand{\cC}{{\mathcal C}}

\newcommand{\cF}{{\mathcal F}}

\newcommand{\cS}{{\mathcal S}}
\newcommand{\cT}{{\mathcal T}}
\newcommand{\cU}{{\mathcal U}}

\newcommand{\g}{\mathfrak{g}}
\newcommand{\n}{\mathfrak{n}}

\newcommand{\bb}{\mathbf{b}}

\newcommand{\be}{\mathbf{e}}
\newcommand{\bi}{{\mathbf i}}

\newcommand{\br}{{\mathbf r}}

\newcommand{\wtB}{\widetilde{B}} 

\newcommand{\LL}{\Lambda}

\newcommand{\rk}{\operatorname{rank}}

\newcommand{\Tor}{\operatorname{Tor}}

\newcommand{\Ima}{\operatorname{Im}}

\newcommand{\ov}{\overline}

\newcommand{\Quot}{\operatorname{Frac}}

\newcommand{\bsm}{\begin{smallmatrix}}
\newcommand{\esm}{\end{smallmatrix}}

\newcommand{\bbsm}{\left[\begin{smallmatrix}}
\newcommand{\besm}{\end{smallmatrix}\right]}

\newcommand{\bbm}{\begin{matrix}}
\newcommand{\ebm}{\end{matrix}}

\newcommand{\Iff}{\Longleftrightarrow}
\newcommand{\ra}{\rightarrow}

\begin{document}

\date{25.07.2018}
\parskip10pt
\setcounter{tocdepth}{1}

\title{Quantum cluster algebras and their specializations}

\author{Christof Gei{\ss}}
\address{Christof Gei{\ss}, Instituto de Matem\'aticas,
Universidad Nacional Aut\'onoma de M\'exico,
Ciudad Universitaria,
04510 Cd.~de M\'exico,
M\'EXICO}
\email{christof.geiss@im.unam.mx}

\author{Bernard Leclerc}
\address{Bernard Leclerc,
Université de Caen Normandie,
CNRS UMR 6139 LMNO, 14032 Caen, FRANCE}
\email{bernard.leclerc@unicaen.fr}

\author{Jan Schr\"oer}
\address{Jan Schr\"oer,
Mathematisches Institut,
Universit\"at Bonn,
Endenicher Allee 60,
53115 Bonn,
GERMANY}
\email{schroer@math.uni-bonn.de}



\begin{abstract}
We show that if a cluster algebra coincides with its upper cluster algebra and 
admits a grading with finite dimensional homogeneous components, the corresponding Berenstein-Zelevinsky 
quantum cluster algebra can be viewed as a flat deformation of the classical cluster algebra. 
\end{abstract}

\maketitle
\tableofcontents



\section{Introduction and main result}


Throughout, let $\K$ be a field of characteristic $0$.
Let $q$ be a variable, and set $R:=\K[q^{\pm 1/2}]$, which is a 
principal ideal domain. 
We write $\K_1 := R/pR$, where $p := q^{1/2} - 1$. 
Of course, as a ring $\K_1\cong \K$. 
However, we want to stress the way how $\K$ is considered as an $R$-module.

We fix some positive integers $m > n$.
Let $(\LL,\wtB)$ be a compatible pair in the sense of \cite[Section~3]{BZ}.
We can see $\LL = (\lambda_{ij}) \in M_m(\Z)$ as a skew-symmetric 
$(m \times m)$-matrix over the integers, and
$\wtB \in M_{m,m-n}(\Z)$ is an $(m \times (m-n))$-matrix over the integers such
that the first $m-n$ rows of $\wtB$ form a skew-symmetrizable matrix, 
which is denoted by $B$.
The matrices $\LL$ and $\wtB$ need to satisfy a compatibility condition.

Let 
$\cT_q(\LL)$
be the $R$-algebra with generators
$X_1,\ldots,X_m,X_1^{-1},\ldots,X_m^{-1}$ subject to the relations
$$
X_iX_i^{-1} = X_i^{-1}X_i =1
\text{\;\;\; and \;\;\;}
X_iX_j = q^{\lambda_{ij}}X_jX_i
$$
for all $1 \le i,j \le m$.
Then $\cT_q(\LL)$ is an Ore domain and can be considered
as a subring of its skew field of fractions $\cF$, compare
\cite{BZ}.
The algebra $\cT_q(\LL)$ is called a \emph{based quantum torus}.
The quantum cluster algebra  $\cA_q := \cA_q(\LL,\wtB)$ is then
a certain $R$-subalgebra of the skew field $\cF$.
(We slightly deviate from the conventions of \cite{BZ} by considering 
the quantum
cluster algebra $\cA_q$ as an $R$-algebra 
and not as a $\Z[q^{\pm 1/2}]$-algebra as in \cite{BZ}.) 

Let $\cA := \cA(\wtB)$ be the (commutative) cluster algebra
(which is a $\K$-algebra) associated with $\wtB$, and let $\cU := \cU(\wtB)$ be its
upper cluster algebra.
The \emph{specialization of} $\cA_q$ for $q = 1$ is defined as
$$
\cA_1 := \K_1 \otimes_R \cA_q.
$$
It is not hard to see that there is a surjective $\K$-algebra homomorphism
\[
\cA_1 \to \cA.
\]
It seems that several authors (including us) 
implicitly assumed that 
the above is an isomorphism.
However, this appears to be far from obvious. 
Up to our knowledge this question was not discussed in the
existing literature. 
The following theorem is a positive result for a restricted class
of quantum cluster algebras.
This fixes a gap in the proof of \cite[Proposition~12.2]{GLS2}.
This proposition is crucial for the proof of the main result 
\cite[Theorem~12.3]{GLS2}.
We thank Alastair King \cite{K} for pointing out the gap.

\begin{Thm}\label{thm:mainresult}
Let $\cA_q$ be a quantum cluster algebra, and let $\cA$ be the
associated (commutative) cluster algebra.
Let $\cU$ be the upper cluster algebra of $\cA$.
We assume the following:
\begin{itemize}

\item[(i)]
$\cA = \cU$;

\item[(ii)]
$\cA$ is a $\Z$-graded cluster algebra with finite-dimensional homogeneous 
components.

\end{itemize}
Then we have 
$$
\K_1 \otimes_R \cA_q \cong \cA.
$$
\end{Thm}

The paper is organized as follows:
After some preparations in Section~\ref{sec2}, we consider
specializations of upper quantum cluster algebras and based quantum tori
in Section~\ref{sec3}.
Section~\ref{sec4} deals with gradings for different kinds of
(quantum) cluster algebras.
The proof of the main result Theorem~\ref{thm:mainresult} is completed in Section~\ref{sec5}.


\section{Divisibility and injectivity of tensor product maps}\label{sec2}


Let $R$ be a (commutative) principal ideal domain, and let 
$p \in R$ be a prime element.
We abbreviate $Q := R/(p)$. 

If $X$ is an $R$-module, we say that $x \in X$ is 
\emph{divisible} by $p$ (in $X$), 
if there exists some $x' \in X$ with $x = p\cdot x'$. 
The following result is probably well known. 
We include the easy proof for the convenience of the reader.

\begin{Lem} \label{lem:inj}
Let $F$ be a free $R$-module, and let $U \subseteq F$ be
a submodule.
Let  $j\df U \to F$ be the inclusion map. 
Then the following are equivalent:
\begin{itemize}

\item[(i)]
The map
\[
Q\otimes_R j\df Q\otimes_R U\ra Q\otimes_R F
\]
is injective;

\item[(ii)]
We have
$\Tor_1^R(Q,F/U) = 0$.

\item[(iii)]
For all $u \in U$ we have
\begin{equation}\label{eq:cond-inj}
   u\text{ is divisible by } p\text{ in } U\Iff u\text{ is divisible by }
    p \text{ in } F
\end{equation}

\end{itemize}
\end{Lem}

\begin{proof} We consider the short exact sequence 
\[
0\ra U\xrightarrow{j} F\rightarrow F/U\rightarrow 0.
\] 
Since $F$ is free,
and in particular flat as an $R$-module, $Q\otimes_R j$ is injective if and
only if 
\[
0 = \Tor_1^R(Q, F/U) \equiv \{ \bar{f}\in F/U \mid 
p \cdot \bar{f} = 0 \},
\]
see for example \cite[Section~3.1, Example~3.1.7]{W}. 
Thus, $\Tor_1^R(Q,F/U)=0$ if and only if
for all $f\in F\setminus U$ we have 
$p\cdot f \notin U$. 

Thus, 
we have to show that the following are equivalent:
\begin{itemize}

\item[(a)]
$\{ f \in F \setminus U \mid p\cdot f \in U \} = \varnothing$;

\item[(b)]
Let $u \in U$ with $u = pf$ for some $f \in F$.
Then $u = pf'$ for some $f' \in U$.

\end{itemize}
Suppose (a) holds.
Then condition (b) becomes empty and is therefore satisfied.
Suppose (b) holds.
Thus let $u \in U$ with $u = pf = pf'$ with $f \in F$ and
$f' \in U$.
This implies $0 = p(f-f')$.
Since $R$ is a domain and $F$ is free, this yields $f = f'$.
Therefore (a) holds.
\end{proof}


\section{Quantum cluster algebras}\label{sec3}


As in the introduction, let
$R := \K[q^{\pm 1/2}]$, $p := q^{1/2}-1$, and $\K_1 := R/pR$.
Let $(\LL,\wtB)$ be a compatible pair, and let $\cT_q(\LL)$, $\cF$ and
$\cA_q := \cA_q(\LL,\wtB)$ be defined as before.

The \emph{initial quantum seed} of $\cA_q$ is denoted by $(M,\wtB)$,
where $M\df \Z^m \to \cF$ is defined
by $M(a_1,\ldots,a_m) := X_1^{a_1} \cdots X_m^{a_m}$ for all
$(a_1,\ldots,a_m) \in \Z^m$. 

Let $(M',\wtB')$ be a quantum seed of $\cA_q$ in the sense of 
\cite[Definition~4.5]{BZ}. 
In particular, $M'$ is a map $\Z^m \to \cF$ such that
$M'(\be_1),\ldots,M'(\be_m)$ is a 
free generating set of $\cF$ and $M'(a_1,\ldots,a_m) = 
M'(\be_1)^{a_1} \cdots M'(\be_m)^{a_m}$ for all $(a_1,\ldots,a_m) \in \Z^m$.
The based quantum torus 
$\cT_qM'$ is a free $R$-module, which has the set 
$\{ M'(c) \mid c \in \Z^m \}$ as an 
$R$-basis. 
 
The following lemma is straightfoward.

\begin{Lem}\label{lem:isom1}
Let $S$ be a commutative ring, and let
$X$ be an $S$-module.
For any ideal $I$ in $S$ there is an isomorphism
\[
\eta_{I,X}\df S/I \otimes_S X \to X/IX
\]
defined by $(s+I) \otimes x \mapsto sx + IX$
for all $s \in I$ and $x \in X$.
\end{Lem}

\begin{Lem}\label{lem:isom2}
For $T := \cT_q(\LL)$ the following hold:
\begin{itemize}

\item[(i)]
$\eta_{pR,T}\df \K_1 \otimes_R T \to T/pT$ is an isomorphism;

\item[(ii)]
The $R$-basis $\{ M(\be) \mid \be \in \Z^m \}$ of $T$ yields a
$\K$-basis $\{ \ov{M}(\be) \mid \be \in \Z^m \}$ of $T/pT$,
where 
$\ov{M}(\be) := \eta_{pR,T}(1 \otimes M(\be))$.

\end{itemize}
\end{Lem}

\begin{proof}
Part (i) is just a special case of Lemma~\ref{lem:isom1}.
We observe that $\{ 1 \otimes M(\be) \mid \be \in \Z^m \}$
is a $\K$-basis of $\K_1 \otimes_R T$.
(Here we use that $T$ is a free $R$-module and that
$\K_1 \otimes_R R \cong \K_1$.)
Now (ii) follows from (i).
\end{proof}

\begin{Lem}\label{lem:surjhom}
With $T = \cT_q(\LL)$ as above, 
the algebra isomorphism
\[
\eta_{pR,T}\df \K_1 \otimes_R T \to T/pT
\]
restricts to a surjective algebra homomorphism
\[
\eta\df \K_1 \otimes_R \cA_q(\LL,\wtB) \to \cA(\wtB).
\]
\end{Lem}

\begin{proof}
Comparing the mutation of quantum seeds of the quantum
cluster algebra $\cA_q(\LL,\wtB)$ as described in
\cite[Proposition~4.9]{BZ} with the mutation of seeds of the
cluster algebra $\cA(\wtB)$, shows that $\eta_{pR,T}$ maps 
the set 
\[
\{ 1 \otimes x_q \mid x_q \text{ is a quantum cluster variable of }
\cA_q(\LL,\wtB) \}
\] 
surjectively onto the set of cluster variables of $\cA(\wtB)$.
\end{proof}

The surjective algebra homomorphism
$\eta$ in Lemma~\ref{lem:surjhom} might not be an isomorphism.
Here one has to keep in mind that the quantum exchange relations 
and the exchange relations might not be the defining relations
of $\cA_q(\LL,\wtB)$ and $\cA(\wtB)$, respectively.
Our aim is to show that under some assumptions on $\cA(\wtB)$,
$\eta$ is indeed an isomorphism.

We follow the convention from Section~\ref{sec2}
  and say that $x\in\cT_q(\LL)$ is \emph{divisible} by $p$  if
$x \in p \cT_q(\LL)$. 


\begin{Lem}\label{lem:div}
Let $T = \cT_q(\LL)$, and
let $x,y\in T$ such that $xy$ is divisible by $p$. 
Then at least one of $x$ and $y$ is divisible by $p$.
\end{Lem}

\begin{proof} 
By assumption we have $xy = pz$ for some $z \in T$.
Furthermore, we know that
\[
T/pT \cong \K_1 \otimes_R T
\cong
\K[x_1^{\pm 1},\ldots,x_m^{\pm 1}].
\] 
So $T/pT$ is a domain, and therefore has no zero 
divisors. 
Now $xy \in pT$ implies that $x \in pT$ or $y \in pT$.
This finishes the proof.
\end{proof}

Let
\[
\cU_q(\LL,\wtB) := \cT_qM \cap \cT_qM_1 \cap \cdots \cap \cT_qM_n
=
\bigcap_{(M',\wtB') \in \cS} \cT_qM',
\]
be the \emph{upper quantum cluster algebra},
where we abbreviated 
$(M_k,\wtB_k) := \mu_k(M,\wtB)$ and $\cS$ denotes
the mutation class of $(M,\wtB)$, compare \cite[Section~5]{BZ}.

The following is our key observation.

\begin{Prop} \label{prop:key}
For any $1 \le k \le n$ we have
\[
((q^{1/2}-1)\cT_qM)\cap \cT_qM_k = \cT_qM\cap ((q^{1/2}-1)\cT_qM_k).
\]
\end{Prop}

\begin{proof}
Without loss of generality, we assume $k=1$.
For $1 \le i \le m$ set 
\[
X_i := M(\be_i) 
\text{\;\;\; and \;\;\;}
X_i' := M_1(\be_i).
\]
Following the proof of \cite[Lemma~4.1]{BFZ} (see also \cite[Lemma~5.5]{BZ}),
every element of $\cT_qM$ can be written uniquely as 
\[
y = \sum_{j=-N}^N X_1^j c_j,
\quad\text{with}\quad 
c_j \in \bigoplus_{\br\in\{0\}\times\Z^{m-1}} 
RM(r_2\be_2+\cdots+ r_m\be_m)
\]
and $N$ large enough.
Then $y$ is divisible by $p$ if and only if
each $c_j$ is divisible by $p$.
Note that the direct sum on the right hand side 
is a subring of $\cT_qM$. 
Now, 
\begin{align*}
X_1&= (X'_1)^{-1}\cdot \left(q^{m_+/2}M(\bb_+) + q^{m_-/2}M(\bb_-)\right)
\\
&=
\left(q^{-m_+/2}M(\bb_+)+q^{-m_-/2}M(\bb_-)\right) \cdot (X'_1)^{-1}
\end{align*}
for certain $\bb_+,\bb_-\in\{0\}\times\N^{m-1}$ and 
$m_+=\be_1^t \LL_1 \bb_+$,
$m_-=\be_1^t \LL_1 \bb_-$. 
For brevity, for $k \in \Z$ we set
\[
Q^{km/2} := q^{km_+/2}M(\bb_+) + q^{km_-/2}M(\bb_-).
\]
Note that 
\begin{align*}
(X_1')^{-1}M(\bb_+) &= q^{-m_+}M(\bb_+)(X_1')^{-1},
\\
(X_1')^{-1}M(\bb_-) &= q^{-m_-}M(\bb_-)(X_1')^{-1}.
\end{align*}
This implies
\[
Q^{km/2}(X_1')^{-1} = (X_1')^{-1}Q^{(k+2)m/2}
\]
for all $k \in \Z$,
and therefore
\begin{align*}
X_1^l &= Q^{-m/2} Q^{-3m/2} \cdots Q^{(1-2l)m/2} (X'_1)^{-l}
\\
&= (X'_1)^{-l} Q^{(2l-1)m/2} Q^{(2l-3)m/2} \cdots Q^{m/2}
\end{align*}
and
\begin{align*}
(X_1')^l &= X_1^{-l}Q^{-m/2} Q^{-3m/2} \cdots Q^{(1-2l)m/2} 
\\
&= Q^{(2l-1)m/2} Q^{(2l-3)m/2} \cdots Q^{m/2}X_1^{-l}
\end{align*}
for all $l > 0$.

Now observe that $y \in \cT_qM_1$ if and only if 
there are some
\[
d_l \in \bigoplus_{\br\in\{0\}\times\Z^{m-1}} RM(r_2\be_2+\cdots r_m\be_m)
\]
such that
\[
y = \sum_{l=-N}^N (X_1')^ld_l.
\]
This yields conditions on the pairs $(c_l,d_{-l})$
for each $-N \le l \le N$.

Using the formulas above (relating $X_1^l$ to $(X_1')^{-l}$),
we get
\[
d_{-l} = Q^{(2l-1)m/2}\cdots Q^{3m/2}Q^{m/2}c_l 
\]
for all $l > 0$, 
and $d_0 = c_0$.
In particular, we have
\[
d_{-l} \in \bigoplus_{\br\in\{0\}\times\Z^{m-1}} RM(r_2\be_2+\cdots r_m\be_m)
\]
for all $0 \le l \le N$.

From the above expressions we see with Lemma~\ref{lem:div} that $c_l$ is 
divisible by $p$ if and only if $d_{-l}$ is divisible by $p$.
This implies our claim
since none of the $Q^{k_i m/2}$ is divisible by $p$.
\end{proof}

Recall that for each $(M',\wtB') \in \cS$, the $\K$-algebra
\[
\K_1\otimes_R\cT_qM' 
\]
is a Laurent polynomial ring in $m$ variables.

\begin{Cor} \label{cor:Uq}
Let $j\df \cU_q(\LL,\wtB)\ra T_qM$ be the natural inclusion.
Then the following hold:
\begin{itemize}

\item[(i)]
The map
\[
\K_1 \otimes_R j\df \K_1\otimes_R \cU_q(\LL,\wtB) \ra 
\K_1\otimes_R \cT_qM
\]
is injective.

\item[(ii)]
Identifying $\cA(\wtB)\subseteq \cU(\wtB)$ naturally with
subalgebras of $\K_1\otimes_R\cT_qM$, we have 
$\cA(\wtB) \subseteq \Ima(\K_1\otimes_R j) \subseteq \cU(\wtB)$.

\end{itemize}
\end{Cor}

\begin{proof} With the help of Proposition~\ref{prop:key} we can show inductively
that each element of $\cU_q(\LL,\wtB)$ fulfills the condition of 
Lemma~\ref{lem:inj}(iii). This shows (i).

For (ii), we observe first that $\cU_q(\LL,\wtB)$ contains all
quantum cluster variables, see \cite[Corollary~5.2]{BZ}.
By \cite[Section~4]{BZ},
quantum cluster variables specialize to
classical cluster variables.
(For a quantum cluster variable $x_q$ let $x := 1 \otimes x_q \in
\K_1 \otimes_R \cT_qM$ be the \emph{specialization} of $x_q$.
One can see $\cA(\wtB)$ as a subalgebra  of $\K_1 \otimes_R \cT_qM$.
It follows from the quantum exchange relations and the
classical exchange relations that $x$ is a cluster variable in 
$\cA(\wtB)$.)
This yields the first inclusion. 
The second inclusion follows directly from the definitions and the
identification of $\K_1 \otimes_R \cT_qM$ with the Laurent
polynomial ring $\cT M$ arising from the initial cluster of $\cA(\wtB)$.
\end{proof}

The following corollary is worth noting, but is not used later on.
We stress that here $\cA_q(\LL,\wtB)$ and $\cU_q(\LL,\wtB)$ are defined
over the Laurent polynomial ring $R$, and not (as for example in
\cite{GY}) over a field (like the field of fractions of $R$).

\begin{Cor}
Suppose that $\cA_q(\LL,\wtB) = \cU_q(\LL,\wtB)$.
Then 
\[
\K_1 \otimes_R \cA_q(\LL,\wtB) \cong \cA(\wtB).
\]
\end{Cor}

\begin{proof}
Combine Corollary~\ref{cor:Uq}(i) with Lemma~\ref{lem:surjhom}.
\end{proof}


\section{Graded cluster algebras}\label{sec4}


Let $\cA$ be a cluster algebra $\cA(\wtB)$ or a quantum cluster algebra
$\cA(\LL,\wtB)$.
Then $\cA$ is a $\Z$-\emph{graded cluster algebra} or
$\Z$-\emph{graded quantum cluster algebra}, respectively,
if the following hold:
\begin{itemize}

\item[(i)]
There is a direct sum decomposition
\[
\cA = \bigoplus_{g \in \Z} \cA_g
\]
such that $\cA_g \cdot \cA_h \subseteq \cA_{g+h}$
for all $g,h \in \Z$;

\item[(ii)]
All cluster variables (resp. quantum cluster variables)
are homogeneous, i.e. for each cluster variable (resp. quantum cluster variable)
$x \in \cA$ there is some $g$ with $x \in \cA_g$.

\end{itemize}

We have $\wtB = (b_{ij}) \in M_{m,m-n}(\Z)$.
Assume that there is some $d = (d_1,\ldots,d_m) \in \Z^n$ 
such that for each $1 \le k \le m-n$ we have
\begin{equation}\label{eq:2}
\sum_{b_{ik} > 0} d_ib_{ik} = \sum_{b_{ik} < 0} d_ib_{ik}.
\end{equation}
Let $X = (X_1,\ldots,X_m)$ be the initial cluster of $\cA$.
Then
\[
\deg_d(X_i) := d_i
\]
extends to all cluster monomials and turns $\cA$ into a $\Z$-graded
cluster algebra with 
\[
 \cA_g := {\rm Span}_{R}(x \mid x\text{ is a product of cluster variables with }
  \deg_d(x) = g \}.
\] 
For the proof one uses the inductive construction of 
cluster variables starting with the initial seed $(M,\wtB)$.
(The initial cluster is then $X = (X_1,\ldots,X_m)$ with 
$X_i = M(\be_i)$ for all $1 \le i \le m$.)
Equation~(\ref{eq:2}) ensures that all exchange relations are
homogeneous. 
We refer to \cite{GL} for more details.

\begin{Lem}\label{lem:graded1}
Suppose $\cA(\wtB)$ is a $\Z$-graded cluster algebra.
Then the grading on $\cA(\wtB)$ induces gradings on
the quantum cluster algebra $\cA_q(\LL,\wtB)$, 
the upper cluster algebra $\cU(\wtB)$, and the upper quantum cluster algebra
$\cU_q(\LL,\wtB)$.
Moreover,
the homogeneous components of 
$\cU_q(\LL,\wtB)$ are free $R$-modules.
\end{Lem}

\begin{proof}
For a $\Z$-graded cluster algebra $\cA(\wtB)$, a corresponding
$\Z$-grading of $\cA_q(\LL,\wtB)$ is constructed analogously,
see \cite{GL}.
As before, 
let $(\LL,M)$ be the initial seed of $\cA_q(\LL,\wtB)$ with
initial cluster $X = (X_1,\ldots,X_m)$ where $X_i := M(\be_i)$ for
$1 \le i \le m$.

By definition, we have
\[
\cU_q(\LL,\wtB) = \bigcap_{(\LL',M') \in \cS} \cT_qM.
\]
The $\Z$-grading on $\cA_q(\LL,\wtB)$ induces a
$\Z$-grading on each quantum torus $\cT_qM'$.
Let
$$
\cT_qM' = \bigoplus_{g \in \Z} (\cT_qM')_g
$$
where the $(\cT_qM')_g$ are the homogeneous components.
Clearly, all $(\cT_qM')_g$ are free $R$-modules (of infinite rank).

Set
\[
U_q(\LL,\wtB)_g := (\cT_qM)_g \cap U_q(\LL,\wtB).
\]
Note that as a submodule of a free $R$-module, $U_q(\LL,\wtB)_g$
is also a free $R$-module.

We claim that
\[
\cU_q(\LL,\wtB) = \bigoplus_{g \in \Z} U_q(\LL,\wtB)_g.
\]
Let $x \in U_q(\LL,\wtB)$.
For each quantum seed $(\LL',M') \in \cS$, it follows that
\[
x = \sum_{g \in \Z} t_{M',g}
\]
for uniquely determined $t_{M',g} \in (\cT_qM')_g$.
Keeping in mind that exchange relations are homogeneous, we get that
$t_{M',g} = t_{M,g}$ for all $(\LL',M') \in \cS$.
This proves our claim.
\end{proof}


\section{Proof of the main result}\label{sec5}


\begin{Thm} \label{thm:main}
Let $(\LL,\wtB)$ be a compatible pair such that 
the following hold:
\begin{itemize}

\item[(i)] 
$\cA(\wtB) = \cU(\wtB)$;

\item[(ii)] 
$\cA(\wtB)$ is a $\Z$-graded cluster algebra with finite-dimensional
homogeneous components.

\end{itemize}
For the initial quantum seed
$(M,\wtB)$ of $\cA_q(\LL,\wtB)$ 
let
\[
i\df\cA_q(\LL,\wtB) \ra \cT_qM
\] 
be the inclusion map.
Then $\K_1\otimes_Ri$ induces an 
isomorphism 
\[
\K_1\otimes_R\cA_q(\LL,\wtB)\cong\cA(\wtB).
\] 
\end{Thm}

\begin{proof}
By Lemma~\ref{lem:surjhom} we have
$\Ima(\K_1\otimes_Ri) = \cA(\wtB)$.
Here we identify $\cA(\wtB)$ again with a subalgebra of
\[
\K_1 \otimes_R \cT_qM \cong \cT M = \K[X_1^\pm,\ldots,X_m^\pm].
\]

By assumption (ii) and Lemma~\ref{lem:graded1},
the upper quantum cluster algebra $\cU_q(\LL,\wtB)$ is a $\Z$-graded 
algebra with all
homogenous components being free $R$-modules.

By the hypothesis (i) and Corollary~\ref{cor:Uq}, the inclusion
\[
j\df \cU_q(\LL,\wtB) \ra \cT_qM
\] 
induces an isomorphism 
\[
\K_1\otimes_R \cU_q(\LL,\wtB) \cong \cA(\wtB) = \cU(\wtB).
\] 
In particular, we have
\[
\rk_R(\cU(\LL,\wtB)_g) = \dim_\K(A(\wtB)_g) < \infty
\] 
for all $g\in\Z$.

Now, consider the inclusion 
\[
i'\df \cA_q(\LL,\wtB) \ra \cU_q(\LL,\wtB).
\]
Since $i = j \circ i'$, by the above remark and Lemma~\ref{lem:surjhom} we have that 
$\K_1 \otimes_R i'$ is surjective. 
Thus if we consider the $R$-module
\[
M := \cU_q(\LL,\wtB)/\cA_q(\LL,\wtB),
\]
we must have $\K_1\otimes_R M = 0$, or in other words $M=pM$,
compare Lemma~\ref{lem:isom1}. 

For each $g \in \Z$ there exists a short exact sequence
\[
0 \to \cA_q(\LL,\wtB)_g \xrightarrow{i'} \cU_q(\LL,\wtB)_g \to M_g \to 0
\]
of finitely generated $R$-modules.
(We know that $\cA_q(\LL,\wtB)_g$ and $\cU_q(\LL,\wtB)_g$
are free $R$-modules of finite rank, in particular both are finitely
generated. 
As a factor module of a finitely generated module, $M_g$ is finitely
generated as well.)
Since $M = pM$ and therefore also $M_g = pM_g$, it follows that $M_g$ does not
have any direct summand isomorphic to $R$  or to $R/(p)$.
This implies that $M_g$ is an $R$-module of finite length.
Therefore the surjective $R$-module homomorphism $M_g \to M_g$
defined by $x \mapsto px$ has to be injective as well.
We conclude that
\[
\Tor_1^R(\K_1,M) = \{ x \in M \mid px=0 \} = 0.
\]
Now Lemma~\ref{lem:inj} implies that
$\K_1\otimes_R i'$ is also injective and therefore an isomorphism.
\end{proof}

Note that Theorem~\ref{thm:main} implies immediately Theorem~\ref{thm:mainresult}.

Let $\Quot(R) := \K(q^{1/2})$ be the field of fractions of
$R = \K[q^{1/2}]$.

\begin{Cor}
With the assumptions of Theorem~\ref{thm:main} and the
notation used in its proof,  
\[
\Quot(R) \otimes i'\df 
\Quot(R) \otimes_R \cA_q(\LL,\wtB) \ra \Quot(R) \otimes_R
\cU_q(\LL,\wtB)
\]
is an isomorphism. 
\end{Cor}

\begin{proof}
In the proof of Theorem~\ref{thm:main} we 
looked at the short exact sequence
\[
0 \to \cA_q(\LL,\wtB)_g \xrightarrow{i'} \cU_q(\LL,\wtB)_g \to M_g \to 0
\]
of $R$-modules for each $g \in \Z$ and saw that
$M_g$ is of finite length.
It follows that the free $R$-modules 
$\cA_q(\LL,\wtB)_g$ and  $\cU_q(\LL,\wtB)_g$ have the same rank, i.e.
\[
\rk_R(\cA_q(\LL,\wtB)_g) = \rk_R(\cU_q(\LL,\wtB)_g).
\]
Since $\Quot(R)$ is a flat $R$-module, 
our claim is equivalent
to $\Quot(R) \otimes_R M_g = 0$ for all $g$.
But this holds, since $M_g$ is of finite length.
\end{proof}

\begin{Qu}
Is it true that
$\cA(\wtB) = \cU(\wtB)$ if and only if
$\cA_q(\LL,\wtB) = \cU_q(\LL,\wtB)$?
(None of the two implications seems to be obvious.)
\end{Qu}

\begin{Conj} 
Hypothesis (ii) in Theorem~\ref{thm:main} is not needed.
\end{Conj}


\section{Examples}


Let $\g = \g(C)$ be a symmetric Kac-Moody Lie algebra associated
with a symmetric generalized Cartan matrix $C$.
For an  element $w$ in the Weyl group of $\g$, let $\n(w)$
be the associated nilpotent Lie algebra, 
and let $N(w)$ be the associated unipotent group,
compare for example \cite[Section~4.3]{GLS1}.
The coordinate ring $\C[N(w)]$ is naturally isomorphic
to a cluster algebra $\cA(w)$, whose initial exchange matrix $\wtB_\bi$
is defined via some reduced expression $\bi = (i_1,\ldots,i_r)$ 
of $w$. 

To the same data, one can associate a $2$-Calabi-Yau Frobenius category
$\cC_w$, which is by definition a subcategory of the category of
finite-dimensional nilpotent modules over a preprojective algebra $\LL$
associated with $C$.
For each $\LL$-module $X \in \cC_w$ there is an evaluation function
$\delta_X$, see \cite[Section~2.2]{GLS1}.

In \cite{GLS1} we proved that $\cA(w)$ is equal to its upper
cluster algebra and that
$\cA(w)$ is isomorphic to a polynomial ring $\C[\delta_{M_1},\ldots,\delta_{M_r}]$, where the $\delta_{M_k}$ are evaluation functions 
arising from a $\LL$-module $M_\bi = M_1 \oplus \cdots \oplus M_r$
in $\cC_w$ associated with $\bi$.
This follows from \cite[Theorem~3.1]{GLS1} and
\cite[Theorem~3.2]{GLS1}.
The map $M_k \mapsto \dim(M_k)$ turns $\cA(w)$ into a $\Z$-graded 
cluster algebra with finite-dimensional homogeneous components.
Here we also used that the exchange relations for $\cA(w)$ arise
from short exact sequences of $\LL$-modules in $\cC_w$, compare
\cite[Sections~2.7 and 2.8]{GLS1}.
This turns the exchange relations into homogeneous relations.

Thus $\cA(w)$ satisfies all assumptions of Theorem~\ref{thm:mainresult}.
We conclude that the specialized quantum cluster algebra
$\C_1 \otimes_R \cA_q(w)$ is isomorphic to $\cA(w)$.
This fixes a gap in the proof of \cite[Proposition~12.2]{GLS2}, which is
essential for the proof of the main result of \cite{GLS2}.


Goodearl and Yakimov \cite{GY}
construct a large class of quantum cluster algebras, called
\emph{quantum nilpotent algebras},
generalizing the quantum cluster algebras $\cA_q(w)$ mentioned above, at least if
we are in the Dynkin case.
They show that these quantum cluster algebras are equal to their
upper quantum cluster algebras.
But note that Goodearl and Yakimov \cite{GY} work over
a field and not (as in our case) over $R = \K[q^{\pm 1/2}]$.
An upcoming manuscript of the same authors will contain an integral version of their results.


\bigskip
{\parindent0cm \bf Acknowledgements.}\,
We thank Alastair King, Travis Mandel, Dylan Rupel and Milen Yakimov
for helpful discussions.
The first named author acknowledges partial support from 
CoNaCyT grant no.~239255 and thanks the Max-Planck Institute for
Mathematics in Bonn for one year of hospitality in 2017/18.
The third author thanks the SFB/TR 45 for financial support.



\begin{thebibliography}{10}

\bibitem[BZ]{BZ}
A. Berenstein, A. Zelevinsky,
\emph{Quantum cluster algebras}.  
Adv. Math.  195  (2005),  no. 2, 405--455.

\bibitem[BFZ]{BFZ}
A. Berenstein, S. Fomin, A. Zelevinsky, 
\emph{Cluster algebras. III. Upper bounds and double Bruhat cells}.  
Duke Math. J.  126  (2005),  no. 1, 1--52.

\bibitem[GLS1]{GLS1}
C. Geiss, B. Leclerc, J. Schr\"oer,
\emph{Kac-Moody groups and cluster algebras}. 
Adv. Math. 228 (2011), no. 1, 329--433. 

\bibitem[GLS2]{GLS2}
C. Geiss, B. Leclerc, J. Schr\"oer,
\emph{Cluster structures on quantum coordinate rings}.
Selecta Math. 19 (2013), no. 2, 337--397.

\bibitem[GY]{GY}
K. Goodearl, M. Yakimov,
\emph{Quantum cluster algebra structures on quantum nilpotent algebras}. 
Mem. Amer. Math. Soc. 247 (2017), no. 1169, vii+119 pp.

\bibitem[GL]{GL}
J. Grabowski, S. Launois, 
\emph{Graded quantum cluster algebras and an application to quantum Grassmannians}.
Proc. London Math. Soc. 109 (2014), 697--732.

\bibitem[K]{K}
A. King, 
Personal communication (Email from 20.07.2017).

\bibitem[W]{W}
C. Weibel,
\emph{An introduction to homological algebra}.
Cambridge Studies in Advanced Mathematics, 38.
Cambridge University Press, Cambridge, 1994. xiv+450 pp.

\end{thebibliography}
\end{document}